\documentclass[11pt]{article}

\usepackage{amsmath,amssymb,amsthm}
\usepackage{color}
\usepackage{graphicx}
\usepackage{url}

\newtheorem{theorem}{Theorem}

\newtheorem{definition}{Definition}
\newtheorem{lemma}{Lemma}

\newtheorem{proposition}{Proposition}
\newtheorem{remark}{Remark}

\def \beq{ \begin{equation}}
\def \eeq{\end{equation}}

\def\sok{{\mathfrak{so}}}

\newcommand{\cn}{\textrm{cn}}
\newcommand{\sn}{\textrm{sn}}
\newcommand{\dn}{\textrm{dn}}

\date{}

\title{ Equivalence of  Rigid Motions and Relative Equilibria in the $N$-Body Problem on the Two-Sphere}

\begin{document}
	\maketitle
	\markboth{T.  Fujiwara, E.  P\'{e}rez-Chavela, and  S. Q.  Zhu}{Rigid Motions and Relative Equilibria}
	\author{\begin{center}
			{ Toshiaki Fujiwara$^1$, Ernesto P\'{e}rez-Chavela$^2$, Shuqiang Zhu$^3$}\\	
			\bigskip
			$^1$ College of Liberal Arts and Sciences, Kitasato University, 1-15-1 Kitasato, Sagamihara, Kanagawa 252-0329, Japan,
fujiwara@kitasato-u.ac.jp\\
			$^2$Department of Mathematics, Instituto Tecnol\'ogico Aut\'onomo\\ de México (ITAM), ernesto.perez@itam.mx\\
			$^3$School of  Mathematics,  Southwestern University of Finance and Economics, Chengdu, China, zhusq@swufe.edu.cn\\
		\end{center}


\begin{abstract}
We investigate the relationship between rigid motions and relative equilibria in the $N$-body problem on the two-dimensional sphere, $\mathbb{S}^2$. We 
prove  that any rigid motion of the $N$-body system on $\mathbb{S}^2$ must be a relative equilibrium. Our approach extends the classical study of rigid body dynamics  and utilizes a rotating frame attached to the particles to derive the corresponding equations of motion. We further show that our results can be extended to the $N$-body gravitational system in $\mathbb{R}^3$.  
The results are oriented to a broader understanding of the dynamics of $N$-body systems on curved surfaces. 
\end{abstract}

\section{Introduction}

Consider the $N$-body problem on $\mathbb{S}^2$.
Let $q_i\in \mathbb{R}^3$ be the vector 
from the centre of $\mathbb{S}^2$ to the particle $i$.
Therefore, $|q_i|=1$. Let the potential $U$ be $SO(3)$ invariant, 
where $SO(3)$ is the group of $3\times 3$ orthogonal    matrices with determinant $1$. We 
will study  rigid motions and relative equilibria of the system.

\begin{definition}\label{def1}
A  rigid motion of the $N$-body problem on $\mathbb{S}^2$ is a solution $q_1(t), \cdots, q_N(t)$  such that $q_i\cdot q_j$ are independent of time  for all pair $(i,j)$, that is, all mutual distances among the particles remain constant along the motion.
\end{definition}

\begin{definition}\label{def2}
A relative equilibrium  of the $N$-body problem  on $\mathbb{S}^2$ is a solution in the form
$q_i(t)=\exp(\xi t) q_i(0)$ for all $i$, where $\xi \in \sok(3)$. 
\end{definition}

Here, $\sok(3)$ is the Lie algebra of 
$SO(3)$,  that is, the set of all $3\times 3$ real anti-symmetric matrices. 
 Then  $\dot{q}_i = \xi q_i$,  and by the well-known correspondence 
 between $\mathfrak{so}(3)$ and $\mathbb{R}^3$, there exists a constant 
 vector $\omega$ such that $\dot{q}_i = \omega \times q_i$ (cf. Appendix \ref{appd:EulerAngles}). In other words, a relative equilibrium is a solution where each particle undergoes a uniform rotation with a constant angular velocity $\omega$.

\begin{remark} 
	If the $z$-axis is chosen to be parallel to $\omega$, then in spherical coordinates ($\theta_i,\varphi_i),$  the motion is described by 
	\begin{equation*} \dot{\theta}_i = 0, \quad \dot{\varphi}_i = |\omega| = \text{constant}, \quad \text{for all } i = 1, 2, \ldots, N. \end{equation*} 
	\end{remark}

Clearly, any relative equilibrium is a rigid motion. Conversely, a rigid motion implies that the motion is merely an isometric rotation. By Euler's rotation theorem, any isometric rotation can be realized by a single rotation around an axis, referred to as the ``instantaneous axis of rotation." Note that the ``instantaneous axis of rotation" is not necessarily fixed. In rigid body dynamics, the rotation of a free (torque-free) rigid body is governed by Euler's equations, which permit complex motions of the rotation axis (see \cite{Arnold1989, Goldstein1980, Landau1960, Whittaker1988}). A uniform rotation, however, is characterized by a fixed (in time) rotation axis, known as the ``permanent axis of rotation" (see \cite{Routh1955}).
 
The goal of this work is to prove that, for the $N$-body problem on $\mathbb{S}^2$, every rigid motion has a ``permanent axis of rotation." More precisely,
\begin{theorem}\label{thm:S2-0}
	Any rigid motion in the $N$-body  problem on $\mathbb{S}^2$
	is a relative equilibrium.
\end{theorem}
 
 In \cite{Borisov2004},  the authors proved that, under the cotangent potential,  any two-body rigid motion on $\mathbb{S}^2$ must be a relative equilibrium.
Some years later in 
\cite{Borisov2018}, the authors extend the above result for arbitrary attractive potentials. 
However, despite the extensive research on relative equilibria of the $N$-body problem on $\mathbb{S}^2$ (cf. \cite{Borisov2004, Diacu2012-3, Diacu2014-2, Zhu2018-2, Perez2023} and references therein), the relationship between relative equilibria and rigid motions has rarely been considered. On the other hand, for the Newtonian $N$-body problem in $\mathbb{R}^3$, this problem has been addressed by  Lagrange \cite{Lagrange1772} for \(N=3\) and by  Pizzetti \cite{Pizzetti} for arbitrary \(N\). 
	Albouy and Chenciner \cite{Albouy1998} studied this problem for the $N$-body problem in $\mathbb{R}^n$.

The focusing in \cite{Borisov2004} and \cite{Borisov2018}  
 does not extend to cases where $N > 2$.    The methods for $\mathbb{R}^n$ 
 	developed by Pizzetti  \cite{Pizzetti}   and Albouy and Chenciner 
 \cite{Albouy1998} do not apply directly
  to problems on $\mathbb{S}^2$ 
 due to differences in the equations of motion, even though  
 Albouy-Chenciner method of reducing the equations of motion 
 is a far-reaching generalization of Lagrange's reduction method 
 for the three-body problem.

In this paper, we address these 
 challenges directly,  using geometric techniques, we solve the problem on the sphere.
Our approach 
is inspired by Euler's work on rigid body dynamics. Given a rigid motion, 
we express the equations of motion in a rotating frame attached to the 
particles. We then derive the corresponding equations governing the 
rotation of the moving frame, which are analogous to Euler's equations. 
As in the case of classical Euler's equations, these equations are 
integrable. We then analyze the types of solutions that can result 
into a rigid motion.  Our analysis  extend easily to that of the 
$N$-body gravitational system in $\mathbb{R}^3 $,  since the two problems share the same symmetry group,  $SO(3)$.

After the introduction, the paper is organized as follows: In Section \ref{sec:setting} we introduce the equations of motion in both the fixed and the rotating frame. In Section \ref{sec:Euler}, we discuss   Euler's equations 
for the rigid motion of the $N$-body problem on $\mathbb{S}^2 $.  In Sections \ref{sec:symmetric} and \ref{sec:asymmetric}, we  analyze which types 
of solutions of Euler's equations  can admit a rigid motion.  
In Section \ref{sec:proof}, we prove Theorem \ref{thm:S2-0}, our main result.   In Section \ref{sec:R3}, we extend 
Theorem \ref{thm:S2-0}
to the  $N$-body problem in $\mathbb{R}^3$  and discuss  Pizzetti's related work.  The Appendix provides some necessary background on Euler angles and solutions of Euler's equations.

\section{Basic settings} \label{sec:setting}
In this section, we derive the equations of motion of the  $N$-body problem on $\mathbb{S}^2$. For rigid motions, we write the equations in a rotating frame connected with the particles.

Consider the $N$-body gravitational system on $\mathbb{S}^2$. Denote by $m_1, \cdots, m_N$ the  positive  masses, and $q_1, \cdots, q_N$ the corresponding positions, where  $q_i \in \mathbb{R}^3 $ for $i=1, \cdots, N$. The system  
is described by the Lagrangian
\beq\label{equ:L_S2-0}
L=\sum_i \frac{m_i}{2}|\dot{q}_i|^2
+U
+\sum_i \lambda_i (|q_i|^2-r_i^2),
\eeq
where $U$ is the potential function that depends on the position
$q_1,q_2,\dots, q_N$ and $SO(3)$ invariant.
That is, 
$$
U=U(q_1,q_2,\dots,q_N)=U(Rq_1,Rq_2,\dots,Rq_N),
$$
for any $R\in SO(3)$.
The $\lambda_i$ in the Lagrangian are the Lagrange multipliers to constrain the bodies on the surface $\left|q_i\right|=r_i$.
Set $r_i=1$ for all $i$ for system on $\mathbb{S}^2$.
But in the following we will take $r_i$ as constant in time,
which can depend on $i$.

\begin{remark}
 For the curved \(N\)-body problem on \(\mathbb{S}^2 \),   $r_i=1$ for all $i$ and  the potential is \(U =\sum_{i\ne j} m_i m_j \frac{q_i \cdot q_j}{\sqrt{1 - (q_i \cdot q_j)^2}}\).
 The potential is known as the cotangent potential, since  \(\frac{q_i \cdot q_j}{\sqrt{1 - (q_i \cdot q_j)^2}}= \cot d_{ij}\), where \( d_{ij}\) is the spherical distance between \(q_i\) and  \(q_j\). 
Thus, the curved \(N\)-body problem on \(\mathbb{S}^2 \) is included in our more general setting.
\end{remark}

The $SO(3)$ invariance yields
\begin{equation}\label{SO3DiffForm}
\sum_i q_i \times \frac{\partial U}{\partial q_i}=0.
\end{equation}
By taking $Rq_i=q_i+\delta\phi \times q_i$,
the invariance takes the form
\[
0
=\sum_i(\delta\phi \times q_i)\cdot\frac{\partial U}{\partial q_i}
=\delta\phi\cdot\left(\sum_i q_i \times \frac{\partial U}{\partial q_i}\right).
\]
Since we can take $\delta\phi$ aribitray, we obtain \eqref{SO3DiffForm}.

For the rigid motion, there is $R\in SO(3)$ with $q(t)=R(t)q(0)$
and an angular velocity $\omega$ with $\dot{q}_i=\omega\times q_i$.
(We will show this in \eqref{defSmallOmega}.)

The  equations of motion are given by:
\beq
\label{equ:equ_S2-0}
m_i \ddot{q}_i=2 \lambda_i q_i+\frac{\partial U}{\partial q_i},
\quad |q_i|^2=r_i^2.
\eeq

Apart from the energy integral, there are two additional  first integrals for the rigid motion.  
Inner product of $\dot{q_i}=\omega\times q_i$ and \eqref{equ:equ_S2-0},
yields the kinetic energy, given by $K=\sum_i (m_i/2) |\dot{q}_i|^2$,  is constant.
Here, we have used \eqref{SO3DiffForm}.
Another conserved quantity is the angular momentum.
Using the identity  $q_i \times q_i=0$ and \eqref{SO3DiffForm},
$c=\sum_i m_i q_i \times \dot{q}_i$ is constant.

\subsection{Rotating frame}\label{GRotatingframe}
To describe rigid motions, it is useful to introduce a rotating coordinate system.  Let $q$ be  the Cartesian radius vector of a point relative to a ``fixed'' (i.e., inertial) system, and $Q$ be  the Cartesian radius vector of the same point relative to a rotating coordinate system. The origins of both  coordinates  system are chosen to coincide with the center of $\mathbb{S}^2 $. The vectors $q, Q$ are related by the following equation
\beq\label{equ:qRQ-0}
q= R(t) Q, 
\eeq
 where $R(t)$ is a twice differentiable curve in $SO(3)$, i.e.,  $R^T=R^{-1}$.
 
 The time derivation yields
 \[ \dot{q}=R\dot{Q}+\dot{R}Q= R(\dot{Q}+R ^{-1} \dot{R}  Q).      \]
 Note that the matrix $A=R ^{-1} \dot{R}$ is anti-symmetric, i.e., $A+ A^T=0$. Let 
 \begin{equation}\label{equ:hat}
 A=\left(\begin{array}{rrr}
 	0\phantom{x} & -\Omega_z & \Omega_y \\
 	\Omega_z & 0\phantom{x} & -\Omega_x \\
 	-\Omega_y & \Omega_x & 0\phantom{x}
 \end{array}\right), \qquad \Omega = (\Omega_x, \Omega_y, \Omega_z)^T. 
 \end{equation}
Then $A\xi=\Omega \times \xi$ for any $\xi \in \mathbb{R}^3 $ (cf. Appendix \ref{appd:EulerAngles}).  Hence, the derivative of $q$ and $Q$ are related by  the identity
\beq\label{equ:derivative_qQ}
\dot{q}= R(\dot{Q}+\Omega \times Q).  
\eeq

For the rigid motion, it holds that $Q_i\cdot Q_j=a_{ij}$. It is convenient to assume that  the rotating coordinate  system is connected to the particles.  Therefore,
$\dot{Q}_i=\ddot{Q}_i=0$, and 
\[\dot q_i=R(\Omega\times Q_i).\]
Similarly, the second derivative is
$$\ddot{q}_i=R\left(\dot\Omega\times Q_i
+\Omega\times(\Omega\times Q_i)\right). $$

Substituting this equation into \eqref{equ:equ_S2-0} we obtain
\beq
\label{equ:equ_S2-2}
\begin{split}
&m_i \left(\dot\Omega\times Q_i +\Omega\times(\Omega\times Q_i)\right)
=2\lambda_i Q_i + \frac{\partial U}{\partial Q_i}.
\end{split}
\eeq

The cross product of $Q_i$ and \eqref{equ:equ_S2-2} yields

\[
m_i \left(
Q_i\times (\dot{\Omega} \times Q_i)
+Q_i \times \left(\Omega\times (\Omega\times Q_i)\right)
\right)
=Q_i \times \frac{\partial U}{\partial Q_i}.
\]

Namely,
\beq\label{equ:equ_S2-3}
m_i Q_i\times (\dot{\Omega} \times Q_i)
+m_i \Omega \times \left(Q_i\times (\Omega\times Q_i)\right)
=Q_i \times \frac{\partial U}{\partial Q_i}.
\eeq
Here, we have used the identity
\[Q_i\times\left(\Omega\times(\Omega\times Q_i)\right)
=\Omega\times\left(Q_i\times(\Omega\times Q_i)\right).\]
The equation \eqref{equ:equ_S2-3} is the main equation.
In the following,
we will use this equation and its outcomes to prove the Theorem \ref{thm:S2-0}.

In the above, we defined $\Omega$ from the anti-symmetric matrix
$R^{-1}\dot{R}$.
Similarly, we can define another anti-symmetric matrix $\dot{R}R^{-1}$,
and the corresponding $\omega$.
Namely,
\begin{equation}\label{defSmallOmega}
\dot{q}_i=\dot{R}Q_i=\dot{R}R^{-1}q_i=\omega\times q_i.
\end{equation}
The $\Omega$ and $\omega$ are related by $\omega=R\Omega$.

\section{Euler's equations}\label{sec:Euler}
In this section, we derive the equations that describe the rotation of the rotating frame, i.e., Euler's equations, and rewrite them in their standard form.

Using \eqref{SO3DiffForm}, 
\eqref{equ:equ_S2-3} yields
\[  \sum_i m_i Q_i\times (\dot{\Omega} \times Q_i)
+ \sum_i m_i \Omega \times \left(Q_i\times (\Omega\times Q_i)\right)=0.      \]
Define 
\[ C=\sum_i m_i Q_i \times (\Omega\times Q_i),  \]
which is the angular momentum in the rotating frame. Since
\beq\label{equ:cC}
 c=\sum_i m_i q_i \times \dot{q}_i= \sum_i m_i (RQ_i) \times (R (\Omega\times Q_i))= R C.  
\eeq
Then the above equation implies that 
\beq\label{equ:Euler-0}
\dot{C}=- \Omega\times C.
\eeq
Note that equation \eqref{equ:Euler-0} is nothing but Euler's equations.

\begin{proposition}\label{Prop:S2-1}
		Given a  rigid motion in the $N$-body gravitational problem on $\mathbb{S}^2$, let the rotating frame be attached to the particles, if the angular velocity \(\Omega\) is a constant vector, then the rigid motion is a relative equilibrium. 
	\end{proposition}
	
	\begin{proof}
		Let \(q_1(t), \cdots, q_N(t)\) denote the rigid motion in the inertial frame and \(Q_1(t), \cdots, Q_N(t)\) the corresponding motion in the rotating frame attached to the particles. Let \(R(t)\) be the rotation matrix in equation \eqref{equ:qRQ-0}. Without loss of generality, assume \(R(0)\) is the identity matrix.
		
		If \(\Omega\) is a constant vector, then the corresponding anti-symmetric matrix \(A = R^{-1} \dot{R}\) is a constant matrix. Consequently, we have
		\[
		R =  \exp(At) \quad \text{and} \quad q_i(t) = \exp(At)  Q_i \text{ for all } i.
		\]
		Thus, the rigid motion is a relative equilibrium.
	\end{proof}

The kinetic energy $K$  in the rotating frame can be written as 
\[ \sum_i \frac{m_i}{2}|\Omega\times Q_i|^2 = \frac{1}{2}\Omega\cdot C.\]
Thus, 
\beq\label{equ:K}
K=\frac{1}{2}\Omega\cdot C=\mbox{ constant}.
\eeq
This fact also yields 
\beq\notag
\dot{\Omega} \cdot C=0, 
\eeq
since  $\Omega\cdot \dot{C}=-\Omega\cdot(\Omega\times C)=0.$

The conservation of the angular momentum $c$  in the inertial frame implies that 
\beq\label{equ:C}
C\cdot C=c\cdot c=\mbox{ constant}.
\eeq
This can also be checked by using $\dot{C}\cdot C=- (\Omega\times C) \cdot C=0$.

Let  $Q_i=(x_i,y_i,z_i)^T$.   
The inertia tensor $I$ is given by 
\[ I=\left( \begin{array}{ccc}
	\sum_i m_i (y_i^2+z_i^2) & -\sum_i m_i x_i y_i & -\sum_i m_i x_i z_i \\
	-\sum_i m_i y_i x_i & \sum_i m_i (z_i^2+x_i^2) & -\sum_i m_i y_i z_i \\
	-\sum_i m_i z_i x_i & -\sum_i m_i z_i y_i & \sum_i m_i (x_i^2+y_i^2)
\end{array}\right). \]
Then the angular momentum $C$ and the angular velocity  $\Omega=(\Omega_x,\Omega_y,\Omega_z)^T$ are related  by the identity
\beq\label{equ:OmegaC}
C=I\Omega, 
\eeq
since 

\begin{align*}
  C=&\sum_i m_i Q_i \times (\Omega\times Q_i)
  = \sum_i m_i \left(|Q_i|^2 \Omega - Q_i  (\Omega\cdot  Q_i)\right)\\
  &= \left(\sum_i m_i \left(|Q_i|^2 E- Q_i     Q_i^T\right)\right)\Omega,  
\end{align*}
where $E$ is the $3\times 3$ identity matrix.

The inertia tensor $I$ is a real symmetric matrix, therefore it 
possesses 
 three eigenvalues and three orthogonal eigenvectors,   known as the \emph{Principal axes}.
By aligning the coordinate axes   $x, y, z$ with these principal axes, the inertia tensor takes the form
\beq \notag 
I=\left(
\begin{array}{ccc}
\sum_i m_i (y_i^2+z_i^2) &0&0 \\
0& \sum_i m_i (z_i^2+x_i^2) &0\\
0&0& \sum_i m_i (x_i^2+y_i^2)
\end{array}\right).
\eeq
Therefore, the eigenvalues are non-negative and at most one of them could be zero. 

In the following, we assume that  the  coordinate axes   $x, y, z$ of the rotating frame are aligned with 
the three eigenvectors of $I$.  
Thus,  the angular momentum is given by
\beq \notag 
C_\alpha = I_\alpha \Omega_\alpha,\  \mbox{where}\  \alpha \in \{ x, y, z\}. 
\eeq
Then Euler's equations \eqref{equ:Euler-0} become
\begin{equation}\label{equ:Euler-1} 
	\begin{split}
		I_x \dot{\Omega}_x&=(I_y-I_z)\Omega_y\Omega_z,\\
		I_y \dot{\Omega}_y&=(I_z-I_x)\Omega_z\Omega_x,\\
		I_z \dot{\Omega}_z&=(I_x-I_y)\Omega_x\Omega_y.
	\end{split}
\end{equation}
Recall that the system admits  two conservative quantities,
the kinetic energy and the angular momentum given by: 
\beq\label{integrals}
\begin{split}
	2K&= \sum_\alpha I_\alpha \Omega_\alpha^2 =  I_x \Omega_x^2+I_y \Omega_y^2+I_z \Omega_z^2,\\
	|C|^2&= \sum_\alpha I_\alpha^2 \Omega_\alpha^2 = I_x^2 \Omega_x^2+I_y^2 \Omega_y^2+I_z^2 \Omega_z^2. 
\end{split}
\eeq

	By Proposition \ref{Prop:S2-1}, Theorem \ref{thm:S2-0}  will be proved once we establish that, for rigid motions, the solution  of  $\Omega$ must be constant.  We begin with a simple case.

\begin{proposition}\label{Prop:all equal}
For any rigid motions in the $N$-body problem on $\mathbb{S}^2 $, 	if all three eigenvalues of $I$ are equal,
then $\Omega$ is a constant vector.
\end{proposition}
\begin{proof} 
Follows immediately from \eqref{equ:Euler-1}. 
\end{proof}

The other cases, where the eigenvalues of $I$ are not all equal, will be discussed in the following sections.

\section{Symmetrical  case}
\label{sec:symmetric}
In this section, we consider the
symmetric case, 
where two eigenvalues of the inertia tensor \(I\) are equal and the third is different. There are two cases to examine, depending on whether one of the eigenvalues is zero or not. 
We will do the corresponding analysis in the following subsections.

\subsection{Symmetric case with one eigenvalue of $I$ being zero}

\begin{proposition}\label{Prop:symmetric-1}
	If one of the eigenvalues of $I$ is zero, 
	then the rigid motion of the $N$-body problem on $\mathbb{S}^2 $ is a relative equilibrium in the inertial frame.  
\end{proposition}

Without loss of generality, we can assume $I_z=\sum m_i (x_i^2+y_i^2)=0$.
This implies $x_i=y_i=0$ for all $N$ particles. 
 
Recall that \(q = RQ\), where \(R \in SO(3)\) is the rotation matrix that relates the inertial frame to the rotating frame attached to the particles. The rotation matrix \(R\) can be decomposed into a product of three simpler rotations, given by
\begin{align*}
R=\left(\begin{array}{ccc}
	\cos\varphi & -\sin\varphi & 0 \\
	\sin\varphi& \cos\varphi & 0 \\
	0 & 0 & 1
\end{array}\right) \left(\begin{array}{ccc}
	1 & 0 & 0 \\
	0 & \cos\theta & -\sin\theta \\
	0 & \sin\theta & \cos\theta
\end{array}\right) \left(\begin{array}{ccc}
	\cos \psi & -\sin\psi& 0 \\
	\sin\psi& \cos\psi& 0 \\
	0 & 0 & 1
\end{array}\right).  
\end{align*}

The three angles $\varphi, \theta, \psi$ are called the \emph{ Euler angles}(see  \cite{Goldstein1980,Landau1960} 
for a geometric meaning of the Euler angles). These angles are related to the  angular velocity in the rotating frame  $\Omega$, see the following equations \eqref{equ:Euler-Kinematics-1} and \eqref{equ:eulerangle-Omega-2}. In the appendix \ref{appd:EulerAngles} you can find a complete derivation of these equations. 

\begin{equation}\label{equ:Euler-Kinematics-1} 
	\begin{split}
	\Omega_x &=\dot\varphi\sin\theta\sin\psi+\dot\theta\cos\psi, \\
	\Omega_y &=	\dot\varphi\sin\theta\cos\psi-\dot\theta\sin\psi, \\
	\Omega_z &=\dot\varphi\cos\theta+\dot\psi. 
	\end{split}
\end{equation}
Take the axes  $x, y, z$ of the rotating frame to be parallel to the principal axes of the inertia $I$. Then 
\beq \label{equ:eulerangle-Omega-2}
I_x \Omega_x = |c| \sin\theta \sin\psi, \ \ I_y \Omega_y = |c| \sin\theta \cos\psi, \  \ I_z\Omega_z = |c| \cos\theta. 
\eeq

\begin{proof}[Proof of Proposition \ref{Prop:symmetric-1}]
	
The eigenvalue $I_z=0$ yields 
\[ I_x=I_y,  \ \ Q_i=(0,0,z_i).  \]
By the third equation of \eqref{equ:eulerangle-Omega-2}, $\cos\theta=0$, 
and then 
$\theta=\pi/2$. So equation \eqref{equ:Euler-Kinematics-1} 
reduces to 
\begin{equation}\notag
		\Omega_x =\dot\varphi\sin\psi, \ \ 
		\Omega_y =	\dot\varphi\cos\psi, \ \ 
		\Omega_z =\dot\psi. 
\end{equation}
Substitute the first equation of the above system into the first equation of \eqref{equ:eulerangle-Omega-2}. Then we get
\[   \dot\varphi = \frac{|c|}{I_x}\Rightarrow  \varphi= \varphi_0 + \frac{|c|}{I_x} t.  \]

In the inertial frame,  the position vector of the particles are given by 
\beq \notag
\begin{split}
	q_i
	&=R Q_i=z_i \left(\begin{array}{ccc}
		\cos\varphi & -\sin\varphi & 0 \\
		\sin\varphi& \cos\varphi & 0 \\
		0 & 0 & 1
	\end{array}\right)   \left(\begin{array}{c}0 \\-1 \\0\end{array}\right)	=z_i\left(\begin{array}{c}
		\sin\varphi \\
		-\cos\varphi \\
		0\end{array}\right)\\
&	=z_i\left(\begin{array}{c}
		\cos( \frac{|c|}{I_x} t +\varphi_0 -\pi/2 ) \\
		\sin(\frac{|c|}{I_x} t +\varphi_0 -\pi/2)\\
		0\end{array}\right).
\end{split}
\eeq
Namely, in the inertia frame,
the particles rotate uniformly, 
 and therefore they form  a relative equilibrium. 
\end{proof}

\subsection{System with positive $I_\alpha$}\label{sec:methodology}

From here on  we focus on the case where \(I_\alpha > 0\) for all \(\alpha = x, y, z\). In this subsection, we provide some fundamental observations that will be useful for our subsequent discussions.

Since $I_\alpha>0$, Euler's equations are
\beq \notag
\dot{\Omega}_\alpha
=I_\alpha^{-1}(I_\beta-I_i)\Omega_\beta\Omega_i,
\eeq
then equation \eqref{equ:equ_S2-3} becomes
\beq\notag
|Q_i|^2\dot{\Omega}-(Q_i \cdot  \dot{\Omega} )Q_i
+(Q_i\cdot\Omega)Q_i\times \Omega
=m_i^{-1}Q_i\times \frac{\partial U}{\partial Q_i}.
\eeq
Or, 
\beq\label{equ:equ_S2-4}
\left(|Q_i|^2 E -Q_i Q_i^T\right)  \dot{\Omega}
+(Q_i\cdot\Omega)Q_i\times \Omega
=v_i, 
\eeq
where $v_i$ is a constant vector in time.

By using Euler's  equations \eqref{equ:Euler-1}, the above system reduces to 
\beq\label{equ:equ_S2-5}
\begin{split}
	&\left(\begin{array}{ccc}
y_i^2+z_i^2 & -x_iy_i & -x_iz_i \\
-y_ix_i & z_i^2+x_i^2 & -y_iz_i \\
-z_ix_i & -z_iy_i & x_i^2+y_i^2
\end{array}\right)
\left(\begin{array}{c}
\frac{I_y-I_z}{I_x}  \Omega_y\Omega_z \\
\frac{I_z-I_x}{I_y} \Omega_z\Omega_x \\
\frac{I_x-I_y}{I_z} \Omega_x\Omega_y
\end{array}\right)\\
&\,\,\,+(x_i \Omega_x+y_i \Omega_y+z_i \Omega_z)
\left(\begin{array}{c}
y_i\Omega_z-z_i\Omega_y \\
z_i\Omega_x-x_i\Omega_z \\
x_i\Omega_y-y_i\Omega_x\end{array}\right)
=v_i.
\end{split}
\eeq

Although in equation \eqref{equ:equ_S2-5} appears three equations, 
only two of them are independent.
Because if we take the inner product with $Q_i$ we obtain that $0=0$, which  is obvious from the form of equation \eqref{equ:equ_S2-3}. 

Note that
each row of the left-hand side is a linear sum of 
$\Omega_x\Omega_y$,
$\Omega_y\Omega_z$,
$\Omega_z\Omega_x$,
$\Omega_x^2$, $\Omega_y^2$, $\Omega_z^2$. 
By using the two first integrals defined in equation \eqref{integrals}, 
if $I_x \ne I_z$, we obtain 
\beq\label{equ:Ox Oz Oy}
\begin{split}
		\Omega_x^2
	&=\frac{2KI_z-|C|^2}{I_x(I_z-I_x)}- \frac{I_y(I_z-I_y)}{I_x(I_z-I_x)}\Omega_y^2,\\
	\Omega_z^2
	&=\frac{|C|^2-2KI_x}{I_z(I_z-I_x)}-\frac{I_y(I_y-I_x)} {I_z(I_z-I_x)}\Omega_y^2. 
\end{split}
\eeq

Now, using equations \eqref{equ:Ox Oz Oy},  the system \eqref{equ:equ_S2-5} takes the form 
$$
c^i_{x y} \Omega_x \Omega_y+c^i_{y z} \Omega_y \Omega_z+c^i_{z x} \Omega_z \Omega_x+c^i_{y y} \Omega_y^2=c^i_0,
$$
where $c^i_{* *}$ and $c^i_0$ are vectors that depend on 
$Q_1, ..., Q_N, K$,
and $|C|^2$. Explicitly, 
\beq \label{coef}
\begin{split}
&c^i_{x y}=\left(\begin{array}{c}
-\frac{\left(I_x-I_y+I_z\right)}{I_z} z_i x_i\\
 \frac{\left(-I_x+I_y+I_z\right)}{I_z} y_i z_i\\
 \left(x_i^2-y_i^2\right)+\frac{\left(I_x-I_y\right)}{I_z}\left(x_i^2+y_i^2\right) 
\end{array}\right), \, 
c^i_{y z}=\left(\begin{array}{c}
(y_i^2-z_i^2)+\frac{I_y-I_z}{I_x}\left(y_i^2+z_i^2\right)\\
-\frac{\left(I_x+I_y-I_z\right)}{I_x} x_i y_i\\
\frac{\left(I_x-I_y+I_z\right)}{I_x} z_i x_i
\end{array}\right), \\
& c^i_{z x}=\left(\begin{array}{c}
\frac{I_x+I_y-I_z}{I_y} x_i y_i\\
 (z_i^2-x_i^2)+ \frac{ \left(I_z-I_x\right)}{I_y}\left(z_i^2+x_i^2\right)\\
 \frac{\left(I_x-I_y-I_z\right)}{I_y} y_i z_i
\end{array}\right), \, 
 c^i_{y y}=\left(\begin{array}{c}
\left(-1- \frac{I_y\left(I_y-I_x\right)}{I_z \left(I_z-I_x\right) } \right) y_i z_i\\
\frac{\left(I^2_x+I^2_z-I_y(I_x+I_z)\right) I_y}{\left(I_x-I_z\right) I_z I_x} z_i x_i\\
\left(1+ \frac{I_y\left(I_y-I_z\right)}{ I_x\left(I_x-I_z\right) } \right) x_i y_i
\end{array}\right). 
\end{split}
\eeq
As shown above,
$c_{**}^i$ depends only on $Q_i=(x_i,y_i,z_i)$ and $I_x,I_y,I_z$.

\begin{definition}
	The functions $f_1, f_2, \dots, f_n$ are called linearly independent
	in an interval $[a,b]$,
	if $c_1f_1(t)+c_2f_2(t)+\dots+c_n f_n(t)=0$
	for constants $c_1, c_2, \dots, c_n$ and
	for all $t\in [a,b]$
	implies $c_1=c_2=\dots=c_n=0$ on this interval.
\end{definition}

Our approach to proving Theorem \ref{thm:S2-0} is as follows: We express the equation \eqref{equ:equ_S2-5} in terms of linearly independent functions of \(t\). To satisfy the equation, all coefficients corresponding to these linearly independent functions must be zero. Since there are several coefficients, we found in a number of conditions that must be met. For example, if the five functions \(\Omega_x \Omega_y\), \(\Omega_y \Omega_z\), \(\Omega_z \Omega_x\), \(\Omega_y^2\), and \(1\) are linearly independent, then all coefficients \(c^i_{**}\) must be zero for all \(i\).

\subsection{Symmetrical  case}

\begin{proposition}\label{Prop:symmetric-2}
	For a  rigid motion of the $N$-body problem on $\mathbb{S}^2 $, if  two of the eigenvalues of $I$ are equal and the third eigenvalue is nonzero, 
	then $\Omega$ is a constant vector.
\end{proposition}

\begin{proof}
Without loss of generality, we can assume 
$I_x=I_y=I_0\ne I_z$ and $I_z\ne 0$.
The Euler's equations \eqref{equ:Euler-1}  are 
\beq \label{equ:Euler-2} 
\begin{split}
I_0 \dot{\Omega}_x&=(I_0-I_z)\Omega_y\Omega_z,\\
I_0 \dot{\Omega}_y&=(I_z-I_0)\Omega_z\Omega_x,\\
I_z \dot{\Omega}_z&=0.
\end{split}
\eeq
Therefore $\Omega_z=$ constant.
Using
\beq \notag 
k=\frac{I_z-I_0}{I_0}\Omega_z = \mbox{ constant},
\eeq
the equations \eqref{equ:Euler-2}  reduce to
\beq \notag 
\dot{\Omega}_x=-k\Omega_y,\,
\dot{\Omega}_y=k\Omega_x,
\eeq
whose solution is
\beq
\Omega_x= a \cos(k t + b),\quad
\Omega_y= a \sin(k t + b),\quad
\Omega_z=\frac{I_0}{I_z-I_0}\, k .
\eeq
with constants $a, b$.

If $ak =0$, that is $a=0$ or $k=0$, then $\Omega_\alpha $ is constant  for all $\alpha=x,y,z$ and therefore $\Omega$ is constant.

If  $ak \ne 0$, denote by $E$ the identity matrix of order 3, 
the system \eqref{equ:equ_S2-5} becomes 
\beq \notag 
\begin{split}
	& k 
\left(|Q_i|^2 E-Q_i Q_i^T\right)
	\left(\begin{array}{c}
	  -\Omega_y\\
		\Omega_x \\
		0
	\end{array}\right)
+ 	 (x_i \Omega_x+y_i \Omega_y+z_i \Omega_z )
	\left(\begin{array}{c}
		y_i\Omega_z-z_i\Omega_y \\
		z_i\Omega_x-x_i\Omega_z   \\
		x_i\Omega_y-y_i\Omega_x\end{array}\right)
	=\left(\begin{array}{c}v_x \\v_y \\v_z\end{array}\right). 
\end{split}
\eeq
The last equation of the above system is 
\beq\notag 
\begin{split}
 &k z_i x_i \Omega_y -k z_i y_i \Omega_x + (x_i^2-y_i^2)\Omega_x \Omega_y+ x_i y_i (\Omega_y^2-\Omega_x^2) \\
 & + \, z_i \Omega_z (x_i \Omega_y-y_i \Omega_x) = v_z.
\end{split}
\eeq

Note that the five functions 
\[ \Omega_x,\,  \Omega_y, \, \Omega_x \Omega_y, \, \Omega_y^2-\Omega_x^2,\,  1     \]
are linearly independent, and $\Omega_z$ is constant.
Thus, 
\[  x_i ^2-y_i^2=0, \ \ \ \, \,  x_iy_i=0.  \]
Namely, $x_i=y_i=0$ for all $k$. 
This makes $I_z=0$, which contradicts $I_z>0$. Thus, it should hold that $ak =0$, and   $\Omega_\alpha $ is constant  for all $\alpha=x,y,z$.
\end{proof}

\section{Asymmetric  case}
\label{sec:asymmetric}
In this section, we consider the asymmetric 
case,  
where the three eigenvalues of \(I\) are distinct. In this case, none of the eigenvalues are zero; otherwise, the remaining two eigenvalues would have to be equal. Without loss of generality, we assume
\beq\notag
0<I_x<I_y<I_z. 
\eeq

From the two first integrals, kinetic energy and angular momentum given in equation \eqref{integrals}, 
we obtain:
\beq \notag 
2KI_x \le |C|^2 \le 2KI_z.
\eeq

First we  consider three special cases. 

\begin{proposition}\label{Prp:asymmetric1}
For a  rigid motion of the $N$-body problem on $\mathbb{S}^2 $, assume that \(0<I_x<I_y<I_z\). 	The angular velocity $\Omega$ is a constant vector in the following three cases
	\begin{enumerate}
		\item[i)]  $|C|^2=2KI_x$;
		\item[ii)] $|C|^2=2KI_z$;
		\item[iii)]  $|C|^2=2KI_y$, and  $2K=I_y \Omega_y^2$. 
	\end{enumerate}
\end{proposition}

\begin{proof}

$i)$ If $|C|^2=2KI_x$, then 
$$I_y(I_y-I_x)\Omega_y^2+I_z(I_z-I_x)\Omega_z^2=0.$$
Namely, $\Omega_y=\Omega_z=0$.
Then Euler's equations \eqref{equ:Euler-1} implies that $\dot \Omega_x=0$. Hence,   $\Omega$ is a constant vector. 

$ii)$ If $|C|^2=2KI_z$, then 
$$I_x(I_z-I_x)\Omega_x^2+I_y(I_z-I_y)\Omega_y^2=0.$$
Namely, $\Omega_x=\Omega_y=0$.
Then Euler's equations \eqref{equ:Euler-1} implies that $\dot \Omega_x=0$. Hence,   $\Omega$ is a constant vector. 

$iii)$ If $|C|^2=2KI_y$
and $2K=I_y \Omega_y^2$,
then 
$I_x^2\Omega_x^2+I_y^2\Omega_y^2+I_z^2\Omega_z^2=I_y^2\Omega_y^2$.
This is realised if and only if
$\Omega_x=\Omega_z=0$.  
Then  Euler's equations \eqref{equ:Euler-1} implies that $\dot \Omega_y=0$. Hence,   $\Omega$ is a constant vector. 
\end{proof}

In the following, we assume that 
\[ 2KI_x < |C|^2 < 2KI_z.   \]
Define $k^2$ as 
\beq\label{defModulus}
k^2=\frac{(I_y-I_x)(2KI_z-|C|^2)}{(I_z-I_y)(|C|^2-2KI_x)}>0.
\eeq
As a function of $|C|^2$,
the denominator is a monotonic increasing function,
and the numerator is monotonic decreasing function.
Therefore, $k^2$ is a monotonic decreasing function.

For $|C|^2=2KI_y$,
\beq \notag 
k^2=\frac{(I_y-I_x)(2KI_z-2KI_y)}{(I_z-I_y)(2KI_y-2KI_x)}=1.
\eeq

  So, the remaining cases can be characterized in term of  $k^2$ by,
\beq \notag
 k^2
    \begin{cases}
  >1&\mbox{ for }|C|^2<2KI_y,\\
  =1&\mbox{ for }|C|^2=2KI_y,\\
  <1&\mbox{ for }|C|^2>2KI_y.
  \end{cases}
  \eeq

In Figure \ref{ellipsoid}, we have drawn some curves for 
 $k^2=1$ and $k^2\ne 1$ 
on the ellipsoid 
$2K = \sum_{\alpha} I_{\alpha} \Omega_{\alpha}^2$.

\begin{figure}[h!]
	\centering
	\includegraphics[scale=0.5]{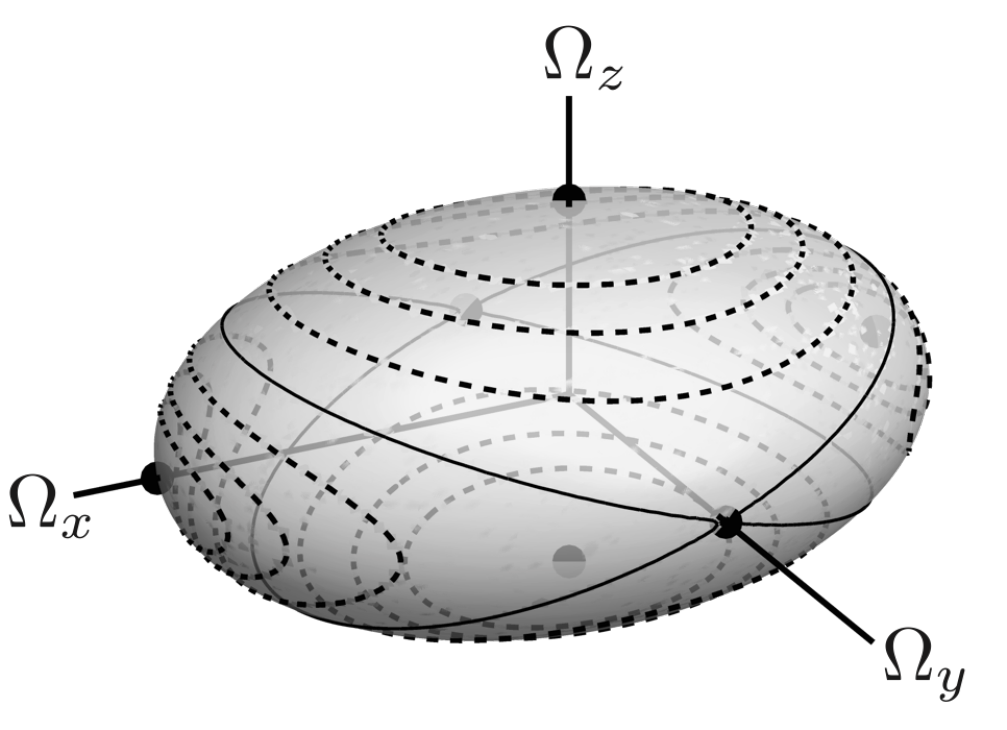}  
	\caption{Ellipsoid for $2 K=\sum_\alpha I_\alpha \Omega_\alpha^2$ for given $K$. The curves represent the contours for $|C|^2=\sum_\alpha I_\alpha^2 \Omega_\alpha^2$ constant.
 The solid, dashed curves represent the contours for 
$k^2=1$, and $k^2\ne 1$, respectively.
The six black balls 
are
the points 
$\Omega=\left(\pm\sqrt{2 K / I_x}, 0,0\right)$
$\left(\mbox{they correspond to } k^2 \rightarrow \infty\right)$,
$\Omega=\left(0, \pm\sqrt{2 K / I_y}, 0\right)$
$\left(k^2=1 \mbox{ and } 2K=I_y\Omega_y^2\right)$,
$\Omega=\left(0,0, \pm\sqrt{2 K / I_z}\right)\left(k^2=0\right)$.
	}\label{ellipsoid}
\notag
\end{figure}

\subsection{Asymmetrical  case  with $0<k^2<\infty$ and $k^2\ne 1$}

In this  subsection we will show that 
\begin{proposition}\label{Prop:asymmetric-2}
	For a  rigid motion of the $N$-body problem on $\mathbb{S}^2 $,  if
	\[  0<I_x<I_y<I_z,  \,  \, 2KI_x < |C|^2 < 2KI_z, \]
	then it is impossible that  $k^2\ne 1$. 
\end{proposition}

\begin{proof}
Perform   the change of variable, 
\beq \label{equ:change variable}
	d\tau=dt \sqrt{\frac{(I_z-I_y)(|C|^2-2KI_x)}{I_xI_yI_z}},
\eeq
and take  $\tau=0$ for $\Omega_y=0$. Then, the solutions for Euler's  equations \eqref{equ:Euler-1} are 
\beq
\begin{split}\notag
	\Omega_x&=\cn(\tau, k)\sqrt{\frac{2KI_z-|C|^2}{I_x(I_z-I_x)}},\\
	\Omega_y&=\sn(\tau,k)\sqrt{\frac{2KI_z-|C|^2}{I_y(I_z-I_y)}},\\
	\Omega_z&=\dn(\tau,k)\sqrt{\frac{|C|^2-2KI_x}{I_z(I_z-I_x)}},
\end{split}
\eeq
where  $\cn(\tau, k), \sn(\tau, k), \dn(\tau, k)$ are Jacobi elliptic functions, and $k$ is the elliptic modulus \cite{Prasolov1997}.  See the appendix \ref{appd:explicit solutions} or \cite{Landau1960} for a complete derivation of them.

In order to finish the proof of Proposition 
\ref{Prop:asymmetric-2}
we need the following result. 

\begin{lemma}\label{Lem:asymmetric-1}
	The functions
	$\Omega_x\Omega_y, \Omega_y\Omega_z,\Omega_z\Omega_x$,
	$\Omega_y^2$, and $1$ are linearly independent, as functions of $t$. 
\end{lemma}
\begin{proof}
	We denote \(\cn(\tau, k)\), \(\sn(\tau, k)\), and \(\dn(\tau, k)\) simply as \(\cn\), \(\sn\), and \(\dn\) respectively.
	Since all constants written after $\cn, \sn, \dn$ are not zero, it suffices to show 
	the linear independence of 
	$\cn\,\sn, \sn\,\dn,\dn\,\cn$, $\sn^2$ and $1$, as functions of $\tau$. 
	Consider the series expansion of
	\beq \notag
	a_1\, \cn\,\sn + a_2\,\sn\,\dn+a_3\,\dn\,\cn+a_4\,\sn^2 + a_5=0
	\eeq
	to order $\tau^4$.
	We get
	\beq
	\sum_{n=0}^4 c_n \tau^n +O(\tau^5)=0. 
	\eeq
	The coefficients $c_n, n=0,1,2,3,4$ are 
		
	$$c_0 =	a_3+a_5, \quad c_1 = a_1+a_2, \quad c_2 = a_4 = -\frac{1}{2} a_3 \left(k^2+1\right),$$
	$$c_3 = -\frac{1}{6} a_1 \left(k^2+4\right)-\frac{1}{6} a_2 \left(4 k^2+1\right),\quad 
	c_4 =	\frac{1}{24} \left(a_3 \left(k^4+14 k^2+1\right)-8 a_4 \left(k^2+1\right)\right)$$

	We can verify that
	the solution for $c_n=0$, $n=0,1,2,3,4$ is 
	$a_1=a_2=a_3=a_4=a_5=0$ for $k^2\ne 1$.
\end{proof}

Now, according to the discussion in Sec \ref{sec:methodology} and  Lemma \ref{Lem:asymmetric-1}, all coefficients vectors $c^i_{**}$ from \eqref{coef}  should be zero. In particular, the third entry of $c^i_{xy}$ and $c^i_{yy}$ is zero, i.e.,  
\beq \notag 
\begin{split}
&\left(x_i^2-y_i^2\right)+\frac{\left(I_x-I_y\right)}{I_z}\left(x_i^2+y_i^2\right) =0,\\
&\left(1+ \frac{I_y\left(I_y-I_z\right)}{ I_x\left(I_x-I_z\right) } \right) x_i y_i=0
\end{split}
\eeq
Namely,  $x_i=y_i=0$ for all $i$.
But, this contradicts with $I_z>0$. 

This finishes the proof of 
Proposition \ref{Prop:asymmetric-2}. 
\end{proof}

\subsection{Asymmetrical case with $k^2=1$ and 
$2K\ne I_y \Omega_y^2$}

In this  subsection, we show that 
\begin{proposition}\label{Prop:asymmetric-3}
	For a  rigid motion of the $N$-body problem on $\mathbb{S}^2 $,  assume that 
	\[  0<I_x<I_y<I_z,  \qquad  |C|^2 =2KI_y. \]
	Then it is impossible that   $2K\ne I_y \Omega_y^2$.
\end{proposition}

\begin{proof}
For asymmetrical 
case 
with $k^2=1$ 
(that 
is for $2KI_y = |C|^2)$, 
we perform the same change of variable that in \eqref{equ:change variable}, 
and take  $\tau=0$ for $\Omega_y=0$. Then we obtain the solutions for Euler's  equations \eqref{equ:Euler-1}, 
\beq
\begin{split}\label{Omegas}
\Omega_x&=\Omega_0 \sqrt{\frac{I_y(I_z-I_y)}{I_x(I_z-I_x)}} \frac{1}{\cosh\tau},\\
\Omega_y&=\Omega_0\tanh\tau,\\
\Omega_z&=\Omega_0\sqrt{\frac{I_y(I_y-I_x)}{I_z(I_z-I_x)}}\frac{1}{\cosh\tau},
\end{split}
\eeq
where $\Omega_0=\frac{|C|}{I_y}=\frac{2K}{|C|}= \sqrt{\frac{2K}{I_y}}.$ See the appendix \ref{appd:explicit solutions} or \cite{Whittaker1988} for a complete derivation of these equations.

Then 
\beq\label{equ:OxOy OyOz}
\begin{split}
\Omega_x\Omega_y&=\sqrt{\frac{I_z(I_z-I_y)}{I_x(I_y-I_x)}}\, \Omega_y\Omega_z\\
\Omega_z\Omega_x&=\sqrt{\frac{(I_y-I_x)(I_z-I_y)}{I_zI_x(I_z-I_x)^2}}\,
(2K-I_y\Omega_y^2).
\end{split}
\eeq
The relations between $\Omega_x^2$,   $\Omega_z^2$ and $\Omega_y^2$
are given by equation \eqref{equ:Ox Oz Oy}.

As in the previous case, to finish the proof of Proposition \eqref{Prop:asymmetric-3} we need the following result. 
\begin{lemma}\label{Lem:asymmetric-2}
	The functions
	$\Omega_y\Omega_z$, 
	$\Omega_y^2$, and $1$ are linearly independent, as functions of $t$. 
\end{lemma}

\begin{proof}
It follows immediately by the equations in \eqref{Omegas}.
\end{proof}

	Now, we decompose  system \eqref{equ:equ_S2-5}  into  linear sums of $\Omega_y\Omega_z$,  $\Omega_y^2$,   and constant. 
Using equations  \eqref{equ:OxOy OyOz}, the coefficient of $\Omega^2_y $ is 
	\[  c^i_{yy}-\frac{I_y}{I_z-I_x} \sqrt{\frac{(I_z-I_y)(I_y-I_x)}{I_x I_z}} c^i_{zx}.  \]
By Lemma \ref{Lem:asymmetric-2}, all three entries of the above vector should be zero. In particular,  
considering the first and the third one and  using\eqref{coef}, we obtain 
	\begin{align*}
	&\left(-1- \frac{I_y\left(I_y-I_x\right)}{I_z \left(I_z-I_x\right) } \right) y_i z_i-\frac{I_y}{I_z-I_x} \sqrt{\frac{(I_z-I_y)(I_y-I_x)}{I_x I_z}}\frac{(I_x+I_y-I_z)}{I_y} x_i y_i=0, \\
	&\left(1+ \frac{I_y\left(I_y-I_z\right)}{ I_x\left(I_x-I_z\right) }\right) x_i y_i-\frac{I_y}{I_z-I_x} \sqrt{\frac{(I_z-I_y)(I_y-I_x)}{I_x I_z}}\frac{\left(I_x-I_y-I_z\right)}{I_y} y_i z_i=0. 
\end{align*}
	
	View the two equations as a  linear system on the two unknowns $x_iy_i, y_iz_i$. The linear system
is equivalent to 
\[  \left(\begin{array}{cc} \alpha (I_x+I_y-I_z) & (I_z-I_x)+ \frac{I_y\left(I_y-I_x\right)}{I_z }\\
	(I_x-I_z)+ \frac{I_y\left(I_y-I_z\right)}{ I_x } ) & \alpha(I_x-I_y-I_z)
	\end{array} \right)
	\left(\begin{array}{c}
			x_iy_i\\y_iz_i
	\end{array}\right)
	=  \left(\begin{array}{c}
		0\\0
	\end{array}\right), 
	\]
	where  $\alpha=\sqrt{\frac{(I_z-I_y)(I_y-I_x)}{I_x I_z}}$. 
	The determinant of the coefficient matrix is 
	\[  2 \frac{I_y  (I_x-I_z)^2 (I_x-I_y+I_z)}{I_x I_z}  \ne 0. \]
	So the linear system has only the zero solution. That is 
	\[  	x_iy_i=y_iz_i=0. \]
	
If $y_i\ne 0$ for some $i$,
then $x_i=z_i=0$ must be satisfied.
Then  the third equation of system  \eqref{equ:equ_S2-5} becomes
\beq
y_i^2 \left(
\frac{I_x-I_y}{I_z}-1
\right)\Omega_x\Omega_y
=\mbox{ constant}.
\eeq
This is impossible since $I_x<I_y<I_z$. 
Therefore $y_i=0$ for all $i$. 
This makes $I_y=I_x+I_z$, which is a contradiction to  $I_x<I_y<I_z$.

This finishes the proof of Proposition \ref{Prop:asymmetric-3}.
\end{proof}

\section{Proof of Theorem \ref{thm:S2-0}} \label{sec:proof}
In the previous sections we have proved all the preliminary results, necessary to establish the proof of the main result of this article.

\begin{proof}[Proof of the main result]

Proposition \ref{Prop:S2-1} tell us that in order to prove Theorem \ref{thm:S2-0}, it is enough to prove that $\Omega$ is a constant vector. In this way we split the analysis in three cases:
\begin{itemize}
\item[1)] All the eigenvalues of the inertia tensor $I$ are equal.

This is the easiest case, it follows immediately from the Euler's equation \eqref{equ:Euler-1}.

\item[2)] Only two of the eigenvalues of $I$ are equal.
\begin{itemize}
\item[i)] If one eigenvalue of $I$ is zero, we proved in Proposition 
 \ref{Prop:symmetric-1} that in the inertial frame, the particles rotate 
 uniformly and they form a relative equilibrium.
 \item[ii)] If no eigenvalue is zero, then in Proposition \ref{Prop:symmetric-2}, we proved that $\Omega$ is a constant vector.
\end{itemize}

\item[3)] If the three eigenvalues of $I$ are distinct and none is zero, we proved first in Proposition \ref{Prp:asymmetric1}, that for three particular cases, Theorem \ref{thm:S2-0} holds and in Propositions 
\ref{Prop:asymmetric-2}, and \ref{Prop:asymmetric-3} we prove the remaining cases.
\end{itemize}
\end{proof}

\section{Rigid motions of the N-body problem in $\mathbb{R}^3 $}\label{sec:R3}

In this section, we extend our method to establish the equivalence of rigid motions with relative equilibria for the \(N\)-body problem in \(\mathbb{R}^3\). This result dates back to Lagrange (1772) for \(N=3\) \cite{Lagrange1772} and was generalized by Pizzetti (1904) for arbitrary \(N\) \cite{Pizzetti, Wintner}.  While this result is classical, our approach differs significantly from that of Pizzetti, as we highlight in this section.

The \(N\)-body problem in \(\mathbb{R}^3\) is described by the Lagrangian:
\[ L = \sum_i \frac{m_i}{2} |\dot{q}_i|^2 + U, \]
where the positive masses are denoted by \(m_1, \dots, m_N\) with positions \(q_1, \dots, q_N\), where \(q_i \in \mathbb{R}^3\). The equations of motion are
\begin{equation}\label{equ:L_R3-0}
m_i \ddot{q}_i=\frac{\partial U}{\partial q_i}, \ i=1, \dots, N.  
\end{equation}
We assume that \(U\) is invariant under translations and rotations, i.e.,
\[
U(q_1,q_2,\dots,q_N)=U(Rq_1+u,Rq_2+u,\dots,Rq_N+u),
\]
for any \(u \in \mathbb{R}^3 ,  R \in SO(3)\). When \(U=\sum_{i\ne j} \frac{m_i m_j}{|q_i-q_j|}\), it corresponds to the Newtonian \(N\)-body problem.

A homographic motion is one where the configuration formed by the \(N\) bodies remains similar to itself for all \(t\), taking the form:
\[  q_i(t)=r(t) + \theta(t)R(t)Q_i, \quad  i=1, \dots, N,  \] 
where \(r(t)\) is a vector function, \(\theta(t)>0\) is a scalar function, \(Q_i=q_i(0)\), and \(R(t) \in SO(3)\). By the center of mass and linear momentum integrals, it suffices to assume
\[
q_i(t)=\theta(t)R(t)Q_i, \quad i=1, \dots, N.
\]

Pizzetti studied homographic motions in the Newtonian \(N\)-body problem \cite{Pizzetti}. By substituting this form into the equations of motion, he derived the equation (in matrix notation):
\begin{equation}\label{equ:pizzetti-1}
\theta^2 R^T(\theta R)''Q_i=  \frac{\partial U}{\partial Q_i}.
\end{equation}  
From this, Pizzetti showed that \(R^TR'\) must be either the zero matrix or a constant nonzero matrix, implying that a homographic motion is either homothetic or has a fixed rotation axis. In the latter case, he further proved that the motion must be planar.

Setting \(\theta(t) \equiv 1\), Pizzetti's result immediately implies that any rigid motion must be a relative equilibrium. Moeckel later used this approach to show the equivalence of rigid motions with relative equilibria for the Newtonian \(N\)-body problem in \(\mathbb{R}^3\) \cite{Moeckel1994}.

While Pizzetti's argument applies to the general \(N\)-body problem \eqref{equ:L_R3-0}, we provide a different proof that offers a more general perspective.

\begin{theorem}\label{thm:R3-0}
	Any rigid motion in the \(N\)-body problem in \(\mathbb{R}^3\) \eqref{equ:L_R3-0} is a relative equilibrium.
\end{theorem}

\begin{proof}  
	 Obviously, the kinetic energy $K$ and the angular momentum $c$ are constant.
	 
	  Assume that  the rigid motion is of the form   $q_i(t)=R(t)Q_i$. Substituting it into \eqref{equ:L_R3-0} yields
\begin{equation} \label{equ: R3-1}
m_i \left(\dot\Omega\times Q_i +\Omega\times(\Omega\times Q_i)\right)
= \frac{\partial U}{\partial Q_i}, 
\end{equation}
where $\Omega= \widehat{R^T R'}$ and $\widehat{*}$ denotes  the hat-map from $\mathfrak{so}(3)$ to \(\mathbb{R}^3\) (cf. Appendix A.2). 
Taking the cross product with \(Q_i\) gives
\begin{equation} \label{equ: R3-2}	 
m_i Q_i\times (\dot{\Omega} \times Q_i)
	 +m_i \Omega \times \left(Q_i\times (\Omega\times Q_i)\right)
	 =Q_i \times \frac{\partial U}{\partial Q_i}. 
	 \end{equation}
	 
		Note that  the above equation \eqref{equ: R3-2} is the same as \eqref{equ:equ_S2-3},  so it leads to 
		Euler's equation \eqref{equ:Euler-1}.  Following the same reasoning as in previous sections, we conclude that the only solution satisfying \eqref{equ: R3-2} is \(\dot\Omega=0\), completing the proof.
\end{proof}

\textbf{Comparison with Pizzetti's method}

At first glance, Pizzetti's method appears similar to ours when setting \(\theta(t) \equiv 1\). However, there is a fundamental difference. Pizzetti's argument  shows that  if $R(t)\in SO(3)$ satisfies $R^T R''Q_i=  \frac{\partial U}{\partial Q_i}$, or equivalently  if
\begin{equation}\label{Pizzetti}
m_i \left(\dot\Omega\times Q_i +\Omega\times(\Omega\times Q_i)\right) =\frac{\partial U}{\partial Q_i}, 
\end{equation}
then \(\Omega\) must be constant.

In contrast, we start from the more general equation
\begin{equation}\label{equ:S2--1}
m_i \left(\dot\Omega\times Q_i +\Omega\times(\Omega\times Q_i)\right) = 2\lambda_i (t) Q_i + \frac{\partial U}{\partial Q_i},
\end{equation}
and derive the equation
\begin{equation}\label{ours}
m_i Q_i \times \left(\dot\Omega\times Q_i +\Omega\times(\Omega\times Q_i)\right) = Q_i \times \frac{\partial U}{\partial Q_i}.
\end{equation}
If $\Omega$ satisfies \eqref{ours}, we show that  $\Omega$ must be constant.   Setting $\lambda_1(t)=\dots =\lambda_N(t)\equiv 0$, our equation \eqref{equ:S2--1} reduces to \eqref{Pizzetti}.  As a result, our approach applies not only to rigid motions in \(\mathbb{R}^3\) but also to those on \(\mathrm{S}^2\), and it may have broader applicability.

While Pizzetti's equation \eqref{Pizzetti}  can be solved directly,  the equation \eqref{ours} requires a different treatment.  We attack it \eqref{ours}  by  integrating key ideas from rigid body dynamics. Specifically, we  utilize  fundamental principles such as 
the conservation of kinetic energy and angular momentum, the inertia tensor, and Euler's equations. To our knowledge, this is the first time that the claim that the axis of rotation for rigid motions on $\mathbb{S}^2$ has been tackled.

The above highlights the originality of our approach.\\

	 \textbf{Acknowledgements}
We sincerely thank the reviewers for their valuable comments and suggestions, especially for pointing out Pizzetti's work, which we were previously unaware of and which has significantly contributed to improving this manuscript.

\appendix

\section{Euler angles}\label{appd:EulerAngles}

Let \( q \) denote the Cartesian radius vector of a point relative to the inertial coordinate system, and let \( Q \) be the Cartesian radius vector of the same point relative to a rotating coordinate system. The origins of both coordinate systems are chosen to coincide with the center of \( \mathbb{S}^2  \). Clearly, the two vectors \( q \) and \( Q \) are related by the following equation:
\[
q = R(t) Q,
\]
where \( R(t) \) is the rotation matrix that transforms coordinates from the rotating frame to the inertial frame. This rotation matrix \( R \) can be decomposed into a product of three simpler rotations. Specifically, let
\[
\begin{split}
	D &= \begin{pmatrix}
		\cos \varphi & -\sin \varphi & 0 \\
		\sin \varphi & \cos \varphi & 0 \\
		0 & 0 & 1
	\end{pmatrix}, \\
	C &= \begin{pmatrix}
		1 & 0 & 0 \\
		0 & \cos \theta & -\sin \theta \\
		0 & \sin \theta & \cos \theta
	\end{pmatrix}, \\
	B &= \begin{pmatrix}
		\cos \psi & -\sin \psi & 0 \\
		\sin \psi & \cos \psi & 0 \\
		0 & 0 & 1
	\end{pmatrix}.
\end{split}
\]
Then \( R \) can be expressed as \( R = DCB \). The angles \(\varphi\), \(\theta\), and \(\psi\) are known as the \emph{Euler angles}. For a detailed discussion on the geometric meaning of the Euler angles, refer to 
\cite{Goldstein1980,Landau1960}.

Chose the inertia frame such that  $c=|c|(0,0,1)^T$.
Then equations \eqref{equ:cC} and \eqref{equ:OmegaC} yield 
\beq \label{equ:eulerangle-Omega-0}
I\Omega= R^T |c|(0,0,1)^T= |c| B^T C^T  D^T (0,0,1)^T. 
\eeq 
Choosing the axes \( x \), \( y \), and \( z \) of the rotating frame to align with the principal axes of the inertia tensor \( I \), equation \eqref{equ:eulerangle-Omega-0} simplifies to
\beq \label{equ:eulerangle-Omega-1}
	I_x \Omega_x = |c| \sin\theta \sin\psi, \ \ I_y \Omega_y = |c| \sin\theta \cos\psi, \  \ I_z\Omega_z = |c| \cos\theta. 
\eeq

\subsection{Relations between $\Omega$ and $\dot{\varphi}, \dot{\theta}, \dot\psi$}
Recall that  the Lie algebra of  $SO(3)$ 
is 
\[  \sok(3)= \{ A \in R^{3\times 3}, A^T=-A \}. \]
That is, $\sok(3)$ consists of all skew-symmetric $3 \times 3$ matrices. 
The Lie algebra $\sok(3)$ (with commutator) is isomorphic to the Lie algebra $\mathbb{R}^3$ (with cross product) by the following hat-map $\sok(3) \to \mathbb{R}^3 $, 
\[  
\widehat{A}=(x,y,z)^T, \ \mbox{for} \  A=\begin{bmatrix}
0&-z&y\\
z&0&-x\\
-y&x&0
\end{bmatrix}\in \sok(3).  
\]
The hat-map has the following identity, 
\begin{equation}\notag
	A u= \widehat{A} \times u, \ \mbox{for}\ \forall u \in \mathbb{R}^3 . 
\end{equation} 

Given any $R\in SO(3), A\in \sok{3}$, note that $RAR^{-1} \in \sok(3)$.   It holds that  $\widehat{RAR^{-1} }=R\hat{A} $.  Because for any $u\in \mathbb{R}^3 $, we have 
\[ (R\hat{A}) \times (Ru)= R ( \hat{A} \times u)= RAu =RAR^{-1} Ru = \widehat{RAR^{-1} } \times (Ru),  \] 
	where we have employed  the formula
\[  \forall M \in GL(3, R), a, b\in \mathbb{R}^3 , (Ma)\times (Mb)= (\det M)  \left(  (M)^{-1}\right)^T (a\times b). \]

Note that $\Omega = \widehat{R^{-1} \dot{R}}$, and 
\begin{align*}
R^{-1}  \dot{R} &= (DCB)^{-1} \dot{D}(CB)+ (DCB)^{-1} D\dot{C}B+(DCB)^{-1} DC \dot{B}\\
                                &= (CB)^{-1} D^{-1}\dot{D}(CB)+ B^{-1} C^{-1}\dot{C}B+ B^{-1}\dot{B}. 
\end{align*}
Then 
\[  \Omega =  (CB)^T  \widehat{D^{-1}\dot{D}} + B^T\widehat{C^{-1}\dot{C}} + \widehat{B^{-1}\dot{B}}.     \]
\begin{align*}
	&D^{-1}\dot{D}= \dot\varphi  \left(\begin{array}{ccc}
		0 & -1 & 0 \\
		1 & 0 & 0 \\
		0 & 0 & 0
	\end{array}\right), ~ & \widehat{D^{-1}\dot{D}}= \dot\varphi \left(\begin{array}{c}0 \\0 \\ 1\end{array}\right);\\
	&C^{-1}\dot{C}= \dot\theta \left(\begin{array}{ccc}
		0 & 0 & 0 \\
		0 & 0 & -1 \\
		0 & 1 & 0
	\end{array}\right), ~& \widehat{C^{-1}\dot{C}}= \dot\theta \left(\begin{array}{c}1\\0 \\0 \end{array}\right);\\
		&B^{-1}\dot{B}= \dot{\psi} \left(\begin{array}{ccc}
		0 & -1 & 0 \\
		1 & 0 & 0 \\
		0 & 0 & 0
	\end{array}\right), ~ & \widehat{B^{-1}\dot{B}}= \dot{\psi} \left(\begin{array}{c}0 \\0 \\ 1\end{array}\right).
\end{align*}
Hence, 
\begin{equation}\label{equ:Euler-Kinematics} 
	\begin{split}
		\Omega_x &=\dot\varphi\sin\theta\sin\psi+\dot\theta\cos\psi, \\
		\Omega_y &=	\dot\varphi\sin\theta\cos\psi-\dot\theta\sin\psi, \\
		\Omega_z &=\dot\varphi\cos\theta+\dot\psi. 
	\end{split}
\end{equation}

\section{Explicit solution to Euler's equations}\label{appd:explicit solutions}
We compute the solutions to Euler's equations \eqref{equ:Euler-1} under the assumption
\[  0< I_x<I_y<I_z, \quad 2KI_x<|C|^2<2KI_z. \]

 The solution $\Omega(t)$ of Euler's  equations
moves on a curve given by \eqref{equ:Ox Oz Oy}.
Some curves (contours) are shown in the Figure \ref{ellipsoid}.
Since the curves are symmetric about the origin,
it is enough to solve the positive branch 
of the square root of \eqref{equ:Ox Oz Oy}.
Namely, 
\beq
\begin{split}\label{eqForOmegaY}
\frac{d\Omega_y}{dt}
=&\frac{I_z-I_x}{I_y}\Omega_z\Omega_x\\
=&\frac{1}{I_y\sqrt{I_zI_x}}\sqrt{(2KI_z-|C|^2)(|C|^2-2KI_x)}\\
	&\sqrt{
		\left(1-\frac{I_y(I_z-I_y)}{2KI_z-|C|^2}\Omega_y^2\right)
		\left(1-\frac{I_y(I_y-I_x)}{|C|^2-2KI_x}\Omega_y^2\right)
	}.
\end{split}
\eeq
So, $\Omega_y$ 
 is 
expressed by Jacobi's elliptic function, in general.

To make the above equation  to the standard form,
let
\beq\label{changevariable}
\begin{split}
s&=\Omega_y \sqrt{\frac{I_y(I_z-I_y)}{2KI_z-|C|^2}},\\
d\tau&=dt \sqrt{\frac{(I_z-I_y)(|C|^2-2KI_x)}{I_xI_yI_z}}.
\end{split}
\eeq
Then, the equation is
\beq \notag
\frac{ds}{d\tau}=\sqrt{(1-s^2)(1-k^2 s^2)}.
\eeq

\textbf{Case I: } For $k^2 \ne 1$, i.e.,   $|C|^2\ne 2KI_y$.

Take $\tau=0$ for $s=0$. Then $s=\sn(\tau, k)$, the Jacobi elliptic sine function with modulus $k$.  Hence, 
\beq 
\begin{split}\label{solForModulusKne1}
\Omega_x&=\cn(\tau, k)\sqrt{\frac{2KI_z-|C|^2}{I_x(I_z-I_x)}},\\
\Omega_y&=\sn(\tau,k)\sqrt{\frac{2KI_z-|C|^2}{I_y(I_z-I_y)}},\\
\Omega_z&=\dn(\tau,k)\sqrt{\frac{|C|^2-2KI_x}{I_z(I_z-I_x)}}.
\end{split}
\eeq

\textbf{Case II: } For $k^2 =1$, i.e.,   $|C|^2=2KI_y$.

	For this case, the equation \eqref{eqForOmegaY} reduces to
	\[  \frac{d\Omega_y}{dt}
	=\frac{2K}{I_y\sqrt{I_zI_x}}\sqrt{(I_z-I_y)(I_y-I_x)}\\
	\left( 1-\frac{I_y\Omega_y^2}{2K}\right). 
	\]
	So $\Omega_y$ is a constant if  $2K=I_y\Omega_y^2$.

The substitutions \eqref{changevariable} reduce to 
\[   s=\sqrt{\frac{I_y}{2K}} \Omega_y,\, \, \, 
d\tau=dt \sqrt{\frac{2K(I_z-I_y)(I_y-I_x)}{I_xI_yI_z}}.  \]
and then 
\beq \notag 
\frac{ds}{d\tau}=1-s^2.
\eeq

Take $\tau=0$ for $s=0$. Then $s=\tanh(\tau)$. 
Hence, 
\beq
\begin{split}\label{OmegaForKeq1}
\Omega_x&=\sqrt{\frac{2K(I_z-I_y)}{I_x(I_z-I_x)}}\,\frac{1}{\cosh(\tau)},\\
\Omega_y&=\sqrt{\frac{2K}{I_y}}\, \tanh(\tau),\\
\Omega_z&=\sqrt{\frac{2K(I_y-I_x)}{I_z(I_z-I_x)}}\,\frac{1}{\cosh(\tau)}.
\end{split}
\eeq


\begin{thebibliography}{10}
	
	\bibitem{Albouy1998}
	A.~Albouy and A.~Chenciner.
	\newblock The {{\(n\)}}-body problem and mutual distances.
	\newblock {\em Invent. Math.}, 131(1):151--184, 1998.
	
	\bibitem{Arnold1989}
	V.~I. Arnold.
	\newblock {\em Mathematical methods of classical mechanics}, volume~60 of {\em
		Graduate Texts in Mathematics}.
	\newblock Springer-Verlag, New York, second edition, 1989.
	\newblock Translated from the Russian by K. Vogtmann and A. Weinstein.
	
	\bibitem{Borisov2004}
	A.~V. Borisov, I.~S. Mamaev, and A.~A. Kilin.
	\newblock Two-body problem on a sphere. {R}eduction, stochasticity, periodic
	orbits.
	\newblock {\em Regul. Chaotic Dyn.}, 9(3):265--279, 2004.
	

	\bibitem{Borisov2018} A.V. Borisov, L.C. Garc\'ia-Naranjo,    I.S. Mamaev,  and J.  Montaldi. \emph{Reduction and relative equilibria for the two-body problem on spaces of constant curvature}, Celestial Mechanics and Dynamical Astronomy, 
{\bf 130}:43, 2018.
	
	\bibitem{Diacu2014-2}
	F.~Diacu.
	\newblock Relative equilibria in the 3-dimensional curved {$n$}-body problem.
	\newblock {\em Mem. Amer. Math. Soc.}, 228(1071):vi+80, 2014.
	
	\bibitem{Diacu2012-3}
	F.~Diacu, E.~P{\'e}rez-Chavela, and M.~Santoprete.
	\newblock The n-body problem in spaces of constant curvature. part i: Relative
	equilibria.
	\newblock {\em J. Nonlinear Sci.}, 22(2):247--266, 2012.
	
	\bibitem{Zhu2018-2}
	F.~Diacu, C.~Stoica, and S.~Q. Zhu.
	\newblock Central configurations of the curved n-body problem.
	\newblock {\em J. Nonlinear Sci.}, 28(5):1999--2046, 2018.
	
	\bibitem{Perez2023}
	T.~Fujiwara and E.~P{\'e}rez-Chavela.
	\newblock Three-body relative equilibria on {{\(\mathbb{S}^2\)}}.
	\newblock {\em Regul. Chaotic Dyn.}, 28(4-5):690--706, 2023.
	
	\bibitem{Goldstein1980}
	H.~Goldstein.
	\newblock {\em Classical mechanics. 2nd ed}.
	\newblock 1980.

	\bibitem{Lagrange1772}
	J.L. Lagrange.
	\newblock Essai sur le probl\`eme des trois corps.
	\newblock {\em {\OE}uvres}, 6:229--324, 1772.

	
	\bibitem{Landau1960}
	L.~D. Landau and E.~M. Lifshits.
	\newblock {\em Mechanics. {Translated} by {J}. {B}. {Sykes} and {J}. {S}.
		{Bell}}.
	\newblock 1960.
	
	\bibitem{Moeckel1994} R.\ Moeckel, {\it Celestial Mechanics---especially central configurations}, unpublished lecture notes: \url{ http://www.math.umn.edu/\~{}rmoeckel/notes/CMNotes.pdf} 
	
	
		\bibitem{Pizzetti}
		P. ~Pizzetti.  
		\newblock Casi particolari del problema dei tre corpi.
		\newblock {\em Rom. Acc. L. Rend. (5)}, 
		 13(1),   17--26, 1904. 
	
	
	\bibitem{Prasolov1997}
	V.~Prasolov and Y.~Solovyev.
	\newblock {\em Elliptic functions and elliptic integrals. {Transl}. from the
		orig. {Russian} manuscript by {D}. {Leites}}, volume 170 of {\em Transl.
		Math. Monogr.}
	\newblock Providence, RI: American Mathematical Society, 1997.
	
	
	\bibitem{Routh1955}
	E.~J. Routh.
	\newblock {\em Treatise on the dynamics of a system of rigid bodies.
		({Advanced} {Part}.) {Sixth} {Edition}, revised and enlarged}.
	\newblock 1955.
	
	\bibitem{Whittaker1988}
	E.~T. Whittaker.
	\newblock {\em A treatise on the analytical dynamics of particles and rigid
		bodies. {With} an introduction to the problem of three bodies. }
	\newblock Camb. Math. Libr. Cambridge (UK) etc.: Cambridge University Press,
	reprint of the fourth edition edition, 1988.
	
	
		\bibitem{Wintner} 
		A. ~Wintner. 
		\newblock {\em The analytical foundations of celestial mechanics.} 
		\newblock Princeton University Press, 1941. 
	
		
	
	
\end{thebibliography}

\end{document}